\documentclass[reqno, 12pt,a4letter]{amsart}
\usepackage{amsmath,amsxtra,amssymb,latexsym, amscd,amsthm,comment}
\usepackage{graphicx,color}
\usepackage{subfigure}
\usepackage[utf8]{inputenc}
\usepackage{enumerate}
\usepackage[mathscr]{euscript}
\usepackage{mathrsfs}
\usepackage[english]{babel}
\usepackage[euler]{textgreek}
\usepackage{color}
\usepackage{cite}

\newtheorem {theorem}{Theorem}[section]
\newtheorem {proposition}[theorem]{Proposition}
\newtheorem {lemma}[theorem]{Lemma}
\newtheorem {corollary}[theorem]{Corollary}
\newtheorem {definition}[theorem]{Definition}
\newtheorem {example}[theorem]{Example}
\newtheorem {remark}[theorem]{Remark}

\def\Limsup{\mathop{{\rm Lim}\,{\rm sup}}}
\def\Sol{\mbox{\rm Sol}\,}

\def\R{\mathbb{R}}
\def\N{\mathbb{N}}

\title{Weak sharp minima at infinity and solution stability in mathematical programming via asymptotic analysis}

\author{Felipe Lara$^1$}
\address {$^1$Instituto de Alta Investigaci\'on (IAI), Universidad de Tarapac\'a, 
Arica, Chile. Web: www.felipelara.cl} 
\email{felipelaraobreque@gmail.com;  flarao@uta.cl}

\author{Nguyen Van Tuyen$^2$}
\address{$^2$Department of Mathematics, Hanoi Pedagogical University 2, Xuan Hoa, Phuc Yen, Vinh Phuc, Vietnam}
\email{nguyenvantuyen83@hpu2.edu.vn; tuyensp2@yahoo.com}

\author{Tran Van Nghi$^3$}
\address{$^3$Department of Mathematics, Hanoi Pedagogical University 2, Xuan Hoa, Phuc Yen, Vinh Phuc, Vietnam}
\email{tranvannghi@hpu2.edu.vn}

\thanks{ }

\date{\today}



\begin{document}
	
\maketitle
	
\begin{abstract}
 We develop sufficient conditions for the existence of the weak sharp minima 
 at infinity property for nonsmooth optimization problems via asymptotic cones 
 and generalized asymptotic functions. Next, we  show that these 
 conditions are also useful for studying the solution stability of nonconvex 
 optimization problems under linear perturbations.  Finally, we provide applications for a subclass of quasiconvex functions which 
 is stable under linear additivity and  includes the convex ones.
\end{abstract}

\section{Introduction}

The notion of asymptotic (or recession) directions of an unbounded set has 
been introduced in order to study its behavior at infinity more than 100 years 
ago in the series of papers \cite{Stein} and then rediscovered in the 1950's 
in connection with some applications in economics \cite{D-1959}. Subsequently, 
the study of asymptotic directions was pursued during decades and the concept 
was developed both for convex and nonconvex sets, and then extended to 
infinite dimensional spaces, too (see 
\cite{Ait,Amara,att-but,Aus1,Aus2,AT,luc89,luc-pen,rock,rock-wet}).

A notion related to the asymptotic cone of the epigraph of the function is the 
so-called asymptotic function. A careful analysis of the behavior of the 
asymptotic function associated to the objective function of the optimization 
problem, along the asymptotic directions of the feasible set, is crucial for 
determining the existence of minimizers (see \cite{AT} for a great account in 
the convex case). But when dealing with nonconvex sets and functions, 
the usual notions of asymptotic cone and function does not provide adequate
information on the asymptotic directions and the level sets of the original 
function, for these reasons, the authors in \cite{Amara,luc89,luc-pen,Penot} 
developed different notions for dealing with nonconvex sets while in 
\cite{Amara,FFB-Vera,HLM,HLL,IL-1,Lara-4,Penot} the authors 
developed different notions for dealing with nonconvex (quasiconvex) functions. 
As a consequence, several notions were introduced, but one of the most 
useful ones in the quasiconvex case is the  so-called {\it $q$-asymptotic 
function}, introduced in \cite{FFB-Vera} (see also \cite{HLM}).

Generalized asymptotic sets and functions have been developed deeply in the 
recent years, and they have been proved to be useful in several nonconvex
optimization problems as, for instance, for developing existence results for 
noncoercive minimization problems (see \cite{Aus1,AT,FFB-Vera,IL-1,Lara-4}),
variational inequalities (see  \cite{AT,IL-4}), equilibrium problems (see 
\cite{Ait,ffb1,IL-3}) and vector optimization problems (see 
\cite{Lara-1,ffb2,F-L-Vera,luc89}) among others.

In this paper, and motivated by very recent developments on different 
theoretical tools for studying minimization problems at infinity (see 
\cite{Kim-Tung-Son-23,Kim-Tung-Son-Tuyen-23,Tung-Son,Tuyen,TBK}), we study weak sharp minima at infinity and 
solution stability for optimization problems via asymptotic analysis. In particular, 
we develop sufficient conditions for weak sharp minima at infinity as well as for 
solution stability under linear perturbations in the general nonconvex case and, 
also, we study continuity properties for the solution map and the optimal value 
function. Furthermore, and by using the $q$-asymptotic function for the 
particular case when the objective function is quasiconvex (i.e., including the 
convex case), we develop even finer sufficient conditions for both mentioned 
problems. Moreover, we prove that our assumptions are weaker than the ones 
used in \cite{Kim-Tung-Son-23} (for weak sharp minima at infinity) for a class 
of nonconvex functions which includes the convex ones, confirming once again 
the importance of the use of asymptotic tools in the study of nonconvex 
optimization problems and, specially, when we are searching for useful 
information from the infinity.

The paper is organized as follows. In Section \ref{section2} we review some 
standard facts on generalized convexity, generalized asymptotic cones and its 
associated asymptotic functions and set-valued mappings. In Section 
\ref{section3}, we provide a finer sufficient condition for characterizing weak 
sharp minima at the infinity. In Section \ref{section4}, we study solution stability 
as well as continuity properties for the solution map and the optimal value 
function. Finally, in Section \ref{section5}, we apply our previous results to the
quasiconvex case by using the $q$-asymptotic function.

\section{Preliminaries and Basic Definitions}\label{section2}

Throughout the paper, the  space $\R^n$ is equipped with the usual scalar 
product $\langle \cdot, \cdot\rangle$ and the corresponding  Euclidean norm 
$\|\cdot\|$. We use the notation  $\mathbb{B}_\delta$ to represent the open 
ball centered at the origin with a radius of $\delta>0.$   The set of all positive integer numbers is denoted by $\N$.
	
Given any function $f:\mathbb{R}^{n}\rightarrow \overline{\mathbb{R}} :=
\mathbb{R}\cup \{ \pm \infty \}$, the effective domain of $f$ is defined by
$$\operatorname{dom}f:=\{x\in \mathbb{R}^{n}:f(x)<+\infty \}.$$ 
We say that $f$ is proper if $f(x)>-\infty$ for every $x\in \mathbb{R}^{n}$ and
$\operatorname{dom}f$ is nonempty. For a function $f$, we adopt the usual
convention $\inf \nolimits_{\emptyset}f:=+\infty$ and $\sup \nolimits_{\emptyset} 
f := -\infty$.
	
	We denote by 
	$$\mathrm{epi}f:=\{(x,t)\in \mathrm{dom}f\times \mathbb{R}:\,f(x)\leq
	t\}$$ 
	its epigraph and for a given $\lambda \in \mathbb{R}$ by 
	$$S_{\lambda}(f) := \{x\in \mathbb{R}^{n}:~f(x)\leq \lambda \}$$ 
	its sublevel set at value
	$\lambda$. As usual,
	$$\mathrm{argmin}_{X}f:=\{x\in X:~f(x)\leq f(y)\ \ \forall y\in X\}.$$
	A proper function $f$ is said to be:
	\begin{itemize}
		\item[$(a)$] {\em convex}  if its domain is convex and for every $x, y \in {\rm dom}\,f,$
		$$f(\lambda x+(1-\lambda)y) \leq \lambda f(x) + (1-\lambda) f(y) \ \ \forall   \lambda \in [0, 1].$$

		\item[$(b)$] {\em quasiconvex} if for every $x, y \in {\rm dom}\,f$,
		$$f(\lambda x + (1 - \lambda)y) \leq \max\{f(x), f(y)\}\ \ \forall  \lambda \in [0,1].$$	
	\end{itemize}
	
	Every convex function is quasiconvex, but the converse statement does not hold as the continuous function $f:\mathbb{R} \rightarrow \mathbb{R}$ with 
	$f(x) := \min \{ \lvert x \rvert,1\}$ shows. Recall that
	\begin{align*}
		f~\mathrm{is~convex} &  \Longleftrightarrow \operatorname{epi}%
		f~\mathrm{is~a~convex~set;}\\
		f~\mathrm{is~quasiconvex} &  \Longleftrightarrow S_{\lambda}%
		(f)~\mathrm{is~a~convex~set,~for~all~}\lambda \in \mathbb{R}.
	\end{align*}
	
As it is well-known, quasiconvex functions are not closed for the sum, i.e., the sum 
of quasiconvex functions is not necessarily quasiconvex. For this reason, the authors 
in \cite{PA} have introduced the following subclass which is closed under addition 
with linear functions.

\begin{definition}[{see \cite{BGJ,PA}}]\label{alpha:robust}   \rm
 For $\alpha \geq 0$, a proper function $f: \mathbb{R}^{n} \rightarrow
 \overline{\mathbb{R}}$ is said to be {\it $\alpha$-robustly quasiconvex} if the 
 function $x \mapsto f(x) + \langle u, x \rangle$ is quasiconvex for all $u \in
 \mathbb{B}_{\alpha}$. 
\end{definition}
Note that every convex function is $\alpha$-robustly quasiconvex for all 
$\alpha \geq 0$, but the converse statements does not holds (see \cite[p. 
1091]{BGJ}). An important property of $\alpha$-robustly quasiconvex functions 
is that every local minimum is global, but $\alpha$-robustly quasiconvex 
functions are different from semistrictly quasiconvex ones (see 
\cite[p. 1091]{BGJ}). For a further study we refer to \cite{BGJ,PA,PA-99}.

As explained in \cite{AT}, the notions of asymptotic cone and the associated 
asymptotic function have been employed in optimization theory in order to handle
unbounded and/or nonsmooth situations, in particular when standard compactness
hypotheses are absent. We recall some basic definitions and properties of 
asymptotic cones and functions, which can be found in \cite{AT}.
	
%\begin{definition}[see \cite{AT}]\rm 
 For a nonempty set $X \subset \mathbb{R}^{n}$, the {\em asymptotic cone} 
 of $X$, denoted by $X^\infty$, is the set defined by
 $$X^{\infty}:=\left \{  u\in \mathbb{R}^{n}:~\exists~t_{k}\rightarrow +
 \infty, ~ \exists~x_{k}\in X,~\frac{x_{k}}{t_{k}}\rightarrow u\right \} .$$
 We adopt the convention that $\emptyset^{\infty}=\emptyset$.    
%\end{definition}	

We note here that when $X$ is a closed and convex set, its asymptotic cone is 
equal to 
\begin{equation*}\label{A1_convex} 
		X^{\infty}=\Big \{u\in \mathbb{R}^{n}:~x_{0}+\lambda u\in
		X\ \ \forall~\lambda \geq0\Big \} \, \text{ for any } x_{0}\in X,
	\end{equation*}
	see \cite[Proposition 2.1.5]{AT}.	

%\begin{definition}[see \cite{AT}]\rm 
 Let $f\colon\mathbb{R}^n\to\overline{\mathbb{R}}$ be a proper function. The 
 {\em asymptotic function} $f^{\infty}: \mathbb{R}^{n}\rightarrow
 \overline{\mathbb{R}}$ of   $f$ is the function for which
\begin{equation*}
 \operatorname{epi}f^{\infty}:=(\operatorname{epi}f)^{\infty}.\label{def:usual}  
\end{equation*}   
%\end{definition}
	From this, one may show that
	\begin{equation*}\label{usual:formulas}
		f^{\infty}(d)=\inf \left \{  \liminf_{k\rightarrow+\infty}\frac{f(t_{k}d_{k})}{t_{k}}:~t_{k}\rightarrow+\infty,~d_{k}\rightarrow d\right \} .
	\end{equation*}
Moreover, when $f$ is lower semicontinuous (lsc henceforth) and convex, then we have
	\begin{equation}
		f^{\infty}(u)=\sup_{t>0}\frac{f(x_{0}+tu)-f(x_{0})}{t}=\lim_{t\rightarrow
			+\infty}\frac{f(x_{0}+tu)-f(x_{0})}{t} \ \ \forall x_{0} \in \operatorname{dom}f, \label{usual:convex} 
	\end{equation}
	see \cite[Proposition 2.5.2]{AT}.	
	
	A function $f\colon\mathbb{R}^n\to\overline{\mathbb{R}}$ is called {\em coercive} on a subset $X\subset\mathbb{R}^n$ if 
	\begin{equation*}
		\lim_{x \overset{X}{\longrightarrow} \, \infty} f(x) = + \infty,
	\end{equation*}
where $x \xrightarrow{X} \infty$ means that $\|x\|\to \infty$ and $x\in X$.	We say that $f$ is coercive if it is coercive on $\mathbb{R}^n$. We know that, if $f^{\infty}(u)>0$ for all $u\neq0$, then $f$ is coercive. Furthermore, if $f$ is convex and lsc, then 
	\begin{equation*}
		f\mathrm{~is~coercive}\Longleftrightarrow f^{\infty}(u)>0\ \ \forall
		 u\neq 0\Longleftrightarrow \mathrm{argmin}_{\mathbb{R}^{n}}~f \ \ \text{is nonempty and compact},\label{char:convex}%
	\end{equation*}
	see \cite[Proposition 3.1.3]{AT}.
	
When $f$ is nonconvex, the asymptotic function $f^{\infty}$ is not good enough 
for providing information on the behavior of $f$. For this reason, several authors 
have been proposed different notions for dealing with, specially, quasiconvex 
functions, see 
	\cite{FFB-Vera,HLM,HLL,IL-1,Lara-4}. 
	\begin{definition}\rm 
		Given a proper function $f: \mathbb{R}^{n} \rightarrow \overline{\mathbb{R}}$, we consider:
		\begin{enumerate}[$(i)$]
			\item (see \cite{FFB-Vera,HLM}) The {\em $q$-asymptotic function}   of $f$ is the function $f^{\infty}_{q}: \mathbb{R}^{n} \rightarrow
			\overline{\mathbb{R}}$ given by:
			\begin{equation*} 
				f^{\infty}_{q} (u) := \sup_{x \in {\rm dom}\,f} \sup_{t>0} \frac{f(x+tu) - f(x)}{t}\ \ \forall ~ u \in \mathbb{R}^{n}.
			\end{equation*}
			\item (see \cite{Lara-4})  The {\em sublevel asymptotic function}  of $f$ at the height $\lambda \in \mathbb{R}$, with $S_{\lambda} (f) \neq \emptyset$, is the
			function $f^{\infty}_{\lambda}: \mathbb{R}^{n} \rightarrow 
			\overline{\mathbb{R}}$ given by:
			\begin{equation*} 
				f^{\infty}_{\lambda} (u) := \sup_{x \in S_{\lambda} (f)} \sup_{t>0} \frac{f(x+tu) - f(x)}{t}\ \ \forall ~ u \in \mathbb{R}^{n}.
			\end{equation*}
		\end{enumerate}   
	\end{definition}	
	
If $f$ is lsc and quasiconvex, then, by \cite[Theorem 4.7]{FFB-Vera}, we have
	\begin{equation*}\label{qasympt:char} f_{q}^{\infty}(u)>0\ \ \forall u\neq0~\Longleftrightarrow
		~\mathrm{argmin}_{\mathbb{R}^{n}}~f\ \ \text{is nonempty and compact},%
	\end{equation*}
	and, by \cite[Theorem 3.1]{Lara-4}, for any $\lambda\in\mathbb{R}$ with $S_{\lambda} (f) \neq \emptyset$  we have 
	\begin{equation*}\label{lambda:char} 
		f^{\infty}_{\lambda} (u) > 0 \ \ \forall   u \neq 0 ~ \Longleftrightarrow ~ ~\mathrm{argmin}_{\mathbb{R}^{n}}~f\ \ \text{is nonempty and compact}.%
	\end{equation*}
	
	If $f$ is quasiconvex (resp. lsc), then $f^{q} (\cdot)$ and $f^{\infty}_{\lambda}
	(\cdot)$ are quasiconvex (resp. lsc). Furthermore, the following
	relations hold for any $\lambda \in \mathbb{R}$ with $S_{\lambda}(f) \neq \emptyset$,
	\begin{equation*} 
		f^{\infty}(\cdot) \leq f^{\infty}_{\lambda} (\cdot) \leq f_{q}^{\infty}(\cdot).
	\end{equation*}
	Both inequalities could be strict even for quasiconvex functions, see, for example,  \cite{Lara-4}. 
	
	We now recall  definitions of the upper and lower semicontinuity to set-valued mappings. 	
	\begin{definition}[see \cite{AT}]\rm 
		Let  $F : \mathbb{R}^n \rightrightarrows \mathbb{R}^m$ be a set-valued mapping. Then, $F $ is said to be:
		\begin{enumerate}[$(i)$]
 \item {\em  upper semicontinuous} (usc henceforth) at $\bar{x}$ if, for any open 
 set $V \subset \mathbb{R}^m$ such that  $F(\bar{x}) \subset V$ there exists a neighborhood $U$ of $\bar{x}$ in $X$ such that  $F(x) \subset V$ for all $x\in U$;
			\item {\em lower semicontinuous} (lsc) at $\bar x$ if $F(\bar x) \neq \emptyset$ and if, for any open set $V \subset \mathbb{R}^m$ such that $F(\bar x) \cap V \neq \emptyset$ there exists a neighborhood $U$ of $\bar x$  such that $F(x) \cap V \neq \emptyset$ for all $x \in U$;
			\item {\em continuous} at $\bar x$ if it is both usc and lsc at this point.
		\end{enumerate}
	\end{definition}	
	
For a further study on generalized convexity and asymptotic analysis we refer to \cite{Amara,att-but,Aus1,Aus2,AT,BGJ,Cambi,ffb1,ffb2,FFB-Vera,HKS,HLL,HLM,HL,IL-1,IL-3,IL-4,IL-5,Lara-1,Lara-4,lara-lopez,luc-pen,PA,Rele-Nedic-24,rock,rock-wet} and references therein.

\section{Weak Sharp Minima at Infinity}\label{section3}
	
Let us consider the following optimization problem
\begin{equation}\label{problem}
 \inf_{x\in X} f(x), \tag{$P$}
\end{equation}
where $f\colon\R^n \to\overline{\R}$ is assumed to be a proper lsc function 
and $X \subset \R^n$ is a closed set such that $\mathrm{dom}\, f\cap X$ is
unbounded during the whole paper.
	
\begin{theorem}[The coercivity and the weak sharp minima property at infinity]\label{weak-sharp-Thrm}  
 Assume that 
		\begin{equation}\label{equa-CQ}
			X^{\infty} \cap \mathcal{K}(f)=\{0\},
		\end{equation}
		where $\mathcal{K}(f):=\{d\in\mathbb{R}^n\,:\, f^\infty(d)\leq 0\}$. 	Then the following statements hold:
		\begin{enumerate}
			\item[$(a)$] Problem \eqref{problem} has a finite optimal value and $\mathrm{Sol}\,\eqref{problem}$ is nonempty and compact.
			
			\item[$(b)$] Problem \eqref{problem} has a weak sharp minima at infinity, i.e., there exist constants $c>0$ and $R>0$ such that
			\begin{equation*}
				f(x)-f_{*} \geq c \, \mathrm{dist}(x,\mathrm{Sol}\,\eqref{problem}) \ \ \forall   x \in X \setminus \mathbb{B}_R,
			\end{equation*}
			where $f_{*} :=\inf_{x\in X} f(x)$ and $\mathrm{dist}(x,\mathrm{Sol}\,\eqref{problem})$ stands for the distance from $x$ to $\mathrm{Sol}\,\eqref{problem}$.
			
			\item[$(c)$] $f$ is coercive on $X$.
		\end{enumerate}
	\end{theorem}
	
	\begin{proof}
		$(a)$: See  \cite[Theorem 4.2.1]{Rele-Nedic-24}.
		
		$(b)$: By contradiction, assume that there exists a sequence $x_k\in X$ such that $\|x_k\|\to\infty$ and
		\begin{equation}\label{equa-3}
			0 \leq f(x_k)-f_*<\frac{1}{k}\mathrm{dist}(x_k,\mathrm{Sol}\,\eqref{problem}) \ \ \forall k\in\N.
		\end{equation}
		Since $\|x_k\|\to\infty$, by passing to subsequences if necessary we may assume that $\frac{x_k}{\|x_k\|}$ converges to some $d\in\R^n$. Clearly, $d\in X^{\infty}$ and  $\|d\|=1$.  By \eqref{equa-CQ},  
		\begin{equation}\label{equa-1}
			f^\infty (d) = \inf\left\{\liminf_{k\to\infty}\frac{f(t_kd_k)}{t_k}\;:\; t_k\to\infty, d_k\to d\right\}>0.
		\end{equation}
		Put $t_k:=\|x_k\|$ and $d_k:=\frac{x_k}{\|x_k\|}$. Then it follows from \eqref{equa-1} that
		\begin{equation}\label{equa-2}
			\gamma:=\liminf_{k\to\infty}\frac{f(t_kd_k)}{t_k}>0.
		\end{equation}
		Let $\varepsilon\in(0, \gamma)$. Then the relation \eqref{equa-2} implies that there exists $k_0>0$ such that
		\begin{equation*}
			\frac{f(t_kd_k)}{t_k}>\gamma-\varepsilon \ \ \forall k\geq k_0,
		\end{equation*}
		or, equivalently,
		\begin{equation*}
			f(x_k)>(\gamma-\varepsilon)\|x_k\|\ \ \forall k\geq k_0.
		\end{equation*}
		This together with \eqref{equa-3} imply that
		\begin{equation*}
			(\gamma-\varepsilon)\|x_k\|-f_*< f(x_k)-f_*<\frac{1}{k}\mathrm{dist}(x_k,\mathrm{Sol}\,\eqref{problem}) \ \  \forall k>k_0.
		\end{equation*}
		Let $z_k\in \mathrm{Sol}\,\eqref{problem}$ be satisfied $\|x_k-z_k\|=\mathrm{dist}(x_k,\mathrm{Sol}\,\eqref{problem})$. Then, by the compactness of $\mathrm{Sol}\,\eqref{problem}$, we have
		\begin{equation*}
			\|x_k-z_k\|\leq \|x_k\|+\|z_k\|\leq \|x_k\|+M,
		\end{equation*}
		where $M:=\max\{\|z\|\;:\; z\in \mathrm{Sol}\,\eqref{problem}\}$. Hence
		\begin{equation}\label{equa-4}
			(\gamma-\varepsilon)\|x_k\|-f_*<\frac{1}{k}(\|x_k\|+M) \ \  \forall k>k_0.
		\end{equation}
		By dividing both sides of \eqref{equa-4} by $\|x_k\|$ and letting
		$k\to\infty$, we get $\gamma\leq \varepsilon$, a contradiction. 
		
		$(c)$: This is a direct consequence of $(a)$ and $(b)$. 
	\end{proof}

Now, we present a result in order to compare with \cite[Theorem 
6.4]{Kim-Tung-Son-23}. To that end, we need to recall some definitions from 
variational analysis at infinity that were introduced in \cite{Kim-Tung-Son-23}. 

\begin{definition}[{see \cite{Kim-Tung-Son-23}}]\label{def31} {\rm 
 The {\em norm cone to the set $X$ at infinity} is defined by
 \begin{eqnarray*}
  N(\infty; X) &:=& \Limsup_{ x \xrightarrow{X} \infty} \widehat{N}(x; X),
 \end{eqnarray*} 
where ``$\displaystyle \Limsup$'' is the sequential  upper/outer limit in the sense of Painlev\'e--Kuratowski, i.e., $u\in  N(\infty; X)$ if and only if there exist sequences 
$x_k \in \R^n$ and $u_k\in\R^n$ such that $u_k\in \widehat{N}(x_k; X)$ for all $k\in\N$ and $\|x_k\|\to \infty$, $u_k\to u$ as 
$k\to\infty$, and $\widehat{N}(x; X)$ is  the regular/Fr\'echet normal cone to $X$ at $x$ 
 and defined by
\begin{align*}
 \widehat N(x; X)=\left\{ v\in \mathbb{R}^n\;:\; \limsup\limits_{z 
 \xrightarrow{X} x} \dfrac{\langle v, z- x \rangle}{\|z- x\|} \leq 0 \right\}.
\end{align*}
}
\end{definition} 

\begin{definition}[{see \cite{Kim-Tung-Son-23}}]\label{def41} {\rm 
 The {\em limiting/Mordukhovich subdifferential} of $f$ at infinity is defined  by 
 \begin{eqnarray*}
 \partial f(\infty) &:=& \Big\{u \in \mathbb{R}^n \ : \ (u, -1) \in \displaystyle
 \Limsup_{x \to \infty} {N}((x,f(x)); \textrm{epi} f) \Big\},
 \end{eqnarray*}
i.e.,    $u\in \partial f(\infty)$ if and only if there exist sequences 
$x_k \in \R^n$, $(u_k, v_k)\in\R^n\times \R$ such that $(u_k, v_k)\in {N}
((x_k,f(x_k)); \textrm{epi} f)$ for all $k\in\N$, and $\|x_k\|\to \infty$, $(u_k, v_k)\to (u, -1)$ as 
$k\to\infty$. 
}\end{definition}

Some properties and calculus rules of the normal cone and the subdifferential at 
infinity can be found in \cite{Kim-Tung-Son-23}. In \cite[Theorem 
6.4]{Kim-Tung-Son-23}, the authors show that if $f$ is bounded from below on 
$X$ and the following condition 
\begin{equation}\label{Son-CQ}
	0\notin \partial f(\infty)+N_X(\infty)
\end{equation}   
is satisfied, then assertions of Theorem \ref{weak-sharp-Thrm} hold. The 
following result gives us  that, in some cases, condition \eqref{equa-CQ} is weaker 
than condition \eqref{Son-CQ}.    

\begin{proposition}\label{prop:comparison2}
Assume that $X$ is a convex set and $f$ is a convex function. 
Then condition \eqref{equa-CQ} is equivalent to that the solution set $\Sol\eqref{problem}$ is nonempty and compact. Consequently, if $f$ is  bounded from below
 on $X$ and  \eqref{Son-CQ} holds, then so is
 \eqref{equa-CQ}. 
\end{proposition}

\begin{proof} If \eqref{equa-CQ} is satisfied, then the nonemptiness and compactness of $\Sol\eqref{problem}$ follow from Theorem \ref{weak-sharp-Thrm}(a). We now assume that $\Sol\eqref{problem}$ is nonempty and compact, but \eqref{equa-CQ} is not satisfied, 
 i.e., there exists a nonzero vector $v \in X^\infty \cap \mathcal{K} (f)$. Fix any $\bar x \in \Sol\eqref{problem}$. Since $f$ is
 proper, lsc, convex and $v \in \mathcal{K} (f)$, we have by formula
 \eqref{usual:convex} that
\begin{equation*}
 f^{\infty} (v) = \sup_{t > 0} \frac{f(x+tv) - f(x)}{t} \leq 0\ \  \forall x \in \mathrm{dom}\,f,
\end{equation*}  
or, equivalently, 
\begin{equation*} f(x+tv) \leq f(x)\ \  \forall   x \in \mathrm{dom}\,f,    \forall   t > 0.
\end{equation*} 
Hence, $ f(\bar{x} + tv) \leq f(\bar{x})$  for all $t > 0.$  Since $X$ is convex and $v\in X^\infty$, one has $\bar x+tv\in X$ for all $t>0$. 
 Thus, $\bar x+tv\in\Sol\eqref{problem}$ for all $t>0$, contrary to the 
 compactness of $\Sol\eqref{problem}$. Therefore, condition \eqref{equa-CQ} 
 is satisfied.
 
 If $f$ is  bounded from below
 on $X$ and  \eqref{Son-CQ} holds, then,   by 
  \cite[Theorem 6.4]{Kim-Tung-Son-23}, $\Sol\eqref{problem}$   is nonempty and compact. Thus \eqref{equa-CQ} is satisfied. The proof is complete.
\end{proof}

\begin{remark}\rm 
If $f$ is an affine function and bounded from below on $X$, then   \eqref{equa-CQ} implies \eqref{Son-CQ} and thus they are  equivalent. Indeed, suppose that \eqref{equa-CQ} is satisfied and  $f(x)=c^Tx+\beta$ for all $x\in\R^n$, where   $c\in\R^n$ and $\beta\in\R$. If \eqref{Son-CQ} does not hold, then $0\in\partial f(\infty)+N_X(\infty)$. Clearly, $\partial f(\infty)=\{c\}$ and so $-c\in N_X(\infty)$. By definition, there exist sequences $x_k\in X$, $x_k^*\in N_X(x_k)$ with $\|x_k\|\to\infty$, $x_k^*\to -c$ as $k\to\infty$. Without any loss of generality, we may assume that $\frac{x_k}{\|x_k\|}\to v\in X^\infty$ with $\|v\|=1$. Since $X$ is convex and $x_k^*\in N_X(x_k)$, we have
\begin{equation*}
	\langle x^*_k, x-x_k\rangle\leq 0\ \ \forall x\in X, \forall k\in\N.
\end{equation*} 
Dividing two sides by $\|x_k\|$ and letting $k\to\infty$ we obtain
\begin{equation*}
	c^Tv=\langle -c, -v\rangle\leq 0.
\end{equation*}
Clearly, $f^\infty(v)=c^Tv$. Hence, $v\in X^\infty\cap\mathcal{K}(f)$ and 
$\|v\|=1$, contrary to \eqref{equa-CQ}. 
\end{remark}

The following simple example shows that \eqref{equa-CQ} is not equivalent 
to \eqref{Son-CQ} for the case of convex quadratic functions.

\begin{example}\rm 
Let $f\colon \R^2\to \R$ and $X\subset\R^2$ be defined, respectively, by
$$f(x):=x_2^2\ \ \forall x=(x_1, x_2)\in\R^2$$
and $X:=\{x=(x_1, x_2)\in\R^2\;:\; 0\leq x_1\leq 1, x_2\geq 0\}$. An easy computation shows that $X^\infty=\{0\}\times\R$ and, for any $v=(v_1, v_2)\in\R^2$,
\begin{equation*}
f^\infty(v)=
\begin{cases}
0, \ \ &\text{if}\ \ v_2=0,
\\
+\infty, \ \ &\text{otherwise}.
\end{cases}
\end{equation*}
Thus $X^\infty\cap\mathcal{K}(f)=\{0\}$. However, we can see that $N_X(\infty)=\R\times\{0\}$, $\partial f(\infty)=\{0\}\times\R$. Hence, 
$0 \in \partial f(\infty)+N_X(\infty)$, i.e., the condition \eqref{Son-CQ} does 
not hold. 
\end{example}

\section{Solution stability}\label{section4}

For every $u\in\R^n$, we define the function $f_u\colon\R^n\to \overline{\R}$ 
by $f_u(x):=f(x)-\langle u, x\rangle$ for all $x\in\R^n$. Consider the perturbed
optimization problem
\begin{equation}\label{problem-u}
 \inf_{x\in X} f_u(x), \tag{$P_u$}
\end{equation}
where $u$ is the parameter of perturbation. The solution set of \eqref{problem-u} 
is denoted by Sol$(u)$. When $u=0$, one has $\mathrm{Sol}\,(0)=\mathrm{Sol}\,
\eqref{problem}$. Furthermore, the function $\mu: \mathbb{R}^n \rightarrow
\overline{\mathbb{R}}$ defined by
\begin{equation*}
 \mu (u)=
 \begin{cases}
			\inf_{x\in X} f_u(x), \;\; \text{if} \;\;\; X \neq 
			\emptyset, \\
			+ \infty, \quad \quad \quad \quad ~ \text{otherwise},
 \end{cases}
\end{equation*}
is said to be the optimal value function of  $(P_{u})$.

Solution stability is an interesting and very useful research field in optimization 
(see \cite{AT,Aus2,HL,rock-wet} among others) in virtue of its applications on 
concrete applications since, in practice, we are usually finding the solution of 
the optimization problem via numerical methods.
	
Before presenting our results, note that: 
	\begin{align}
		& (f_{u})^{\infty} (y) = (f)^{\infty} (y) - \langle u, y \rangle, \label{equa-5a} \\
		& (f_{u})^{\infty}_{q} (y) = (f)^{\infty}_{q} (y) - \langle u, y \rangle. \label{equa-5q}
	\end{align}
	The proofs are similar, we just prove \eqref{equa-5a}. For every
	$y, u \in \mathbb{R}^{n}$, we have
	\begin{align*}
		(f_{u})^{\infty} (y) = \liminf\limits_{\mathop {y^\prime \to y}\limits_{t \to \infty}} \frac{f(t y^\prime) - \langle u, t y^\prime \rangle}{t} & = \liminf\limits_{\mathop {y^\prime \to y} \limits_{t \to \infty}} \frac{f(t y^\prime)}{t} + \liminf\limits_{\mathop {y^\prime \to y}\limits_{t \to \infty}} (-\langle u, y^\prime \rangle) \notag
		\\
		& = (f)^{\infty} (y) - \langle u, y \rangle.
	\end{align*}

\subsection{The semicontinuity of the solution map}
	
	\begin{theorem}\label{theo:sta01} 
		Assume that $X^{\infty} \cap \mathcal{K}(f)=\{0\}.$	 Then there exists $\varepsilon>0$ such that for all $u \in \mathbb{B}_\varepsilon$, the following statements hold:
		\begin{enumerate} [$(a)$]
			\item $f_u$ is bounded from below on $X$.
			
			\item  $f_u$ is coercive.
			
			\item  $\Sol(u)$ is nonempty and compact.
			
			\item  $\Limsup_{u \to 0}\Sol(u)\subset \Sol(0)$.
			
			\item   $\Sol(\cdot)$ is usc at $0$.
		\end{enumerate}	
	\end{theorem}	
	\begin{proof} 
		We first show that there exists $\varepsilon>0$ such that for all $u\in\mathbb{B}_\varepsilon$, the following condition holds 
		\begin{equation}\label{equa-6}
			X^{\infty} \cap\mathcal{K}(f_u)=\{0\}.
		\end{equation}  
		Indeed, if otherwise, then for any  $k\in\N$, there is $u_k\in \mathbb{B}_{\frac{1}{k}}$ such that $X^{\infty} \cap \mathcal{K}(f_{u_k})\neq \{0\}$, i.e., there exists $d_k\in X^{\infty} \setminus\{0\}$ such that $(f_{u_k})^\infty (d_k)\leq 0$. Since $X^{\infty}$ is a closed cone, by passing a subsequence if necessary we may assume that $h_k:=\frac{d_k}{\|d_k\|}$ converges to some $h\in X^{\infty}$ with $\|h\|=1$. For each $k>0$, by the positive homogeneity of $(f_{u_k})^\infty$ and the fact that $(f_{u_k})^{\infty} (d_k)\leq 0$, one has $(f_{u_k})^{\infty} (h_k)\leq 0$.		By  \eqref{equa-5a}, we have
		\begin{equation*} 
			f^\infty (h_k) - \langle u_k, h_k \rangle= (f_{u_k})^{\infty} (h_k) \leq 0,
		\end{equation*}
		or, equivalently,
		\begin{equation*} 
			f^\infty (h_k)\leq \langle u_k, h_k \rangle \ \ \forall k\in\N.
		\end{equation*}
		This and the lower semicontinuity of $f^\infty$  imply that
		\begin{equation*}
			f^{\infty} (h) \leq \liminf_{k \to \infty}f^{\infty} (h_k) \leq \lim_{k \to \infty} \langle u_k, h_k\rangle=0.  
		\end{equation*}
		Hence $h\in X^{\infty} \cap \mathcal{K}(f)$, which contradicts the assumption that $X^{\infty} \cap \mathcal{K}(f)=\{0\}$. Thus \eqref{equa-6} holds. 
		
		$(a)$, $(b)$, and $(c)$ follow directly from Theorem \ref{weak-sharp-Thrm} and \eqref{equa-6}.
		
		$(d)$: Take any $\bar x\in \Limsup_{u\to 0}\mathrm{Sol}\,(u)$. Then there exist sequences $u_k\to 0$ and $x_k\in \mathrm{Sol}\,(u_k)$ with $x_k\to \bar x$ as $k\to\infty$. For each $k>0$, we have
		\begin{equation*}
			f(x_k)-\langle u_k, x_k\rangle\leq f(x)- \langle u_k, x\rangle  \ \  \forall x\in X.
		\end{equation*}
		This and the lower semicontinuity of $f$ imply that
		\begin{align*}
			f(\bar x) \leq \liminf_{k \to \infty} f(x_k) & = \liminf_{k \to \infty} (f(x_k) - \langle u_k, x_k \rangle) \\
			& \leq \liminf_{k \to \infty} (f(x) - \langle u_k, x \rangle) = f(x) \ \ \forall x \in X.
		\end{align*}
 Hence $\bar x\in \Sol(0)$, as required. 
		
		$(e)$:  Suppose on the contrary that  $\Sol(\cdot)$ is not usc at $0$. Then there exists an open set $U \subset \mathbb{R}^n$, with $\Sol(0) \subset U$, such that for every neighborhood $0 \in W \subset \mathbb{R}^n$, there exists  $u\in W$ satisfying  $\Sol(u)\nsubseteq U$. Hence, there exist   sequences $u_k$ and $x_k$ such that  $u_k\rightarrow 0$ and  $x_k \in \Sol(u_k) \setminus U$ for all $k \in \mathbb{N}$.  
		
		If  $\{x_k\}$ is bounded, we assume without loss of generality  that $x_k \to \hat{x} \in \mathbb{R}^n$.  By (d),  $\hat{x} \in \Sol(0) \subset U$, which contradicts that $x_k \notin U$ for all $k$ and $U$ is
		open. So, $\{x_k\}$ is unbounded.  Without any loss
		of generality, we can assume that $\|x_k\|\to\infty$ as $k\to\infty$. 	Take $t_k := \|x_k\|$ and $d_k := \frac{x_k}{t_k}$. Hence, $d_{k}
		\to d \in X^{\infty}$ with $\|d\| = 1$, taking a subsequence if necessary. Fix any $y \in X$. 
		Then, we have $f_{u_k} (x_k) \leq f_{u_k}(y)$ for all $k$  and  so
		\begin{align*}
			f^\infty (d) & \leq \liminf_{k \rightarrow \infty} \dfrac{f(t_kd_k)}{t_k} = \liminf_{k \rightarrow \infty} \dfrac{f(x_k) - \langle u_k, x_k \rangle}{t_k} \leq \liminf_{k \rightarrow \infty} \dfrac{f(y) - \langle  u_k, y \rangle}{t_k} \, = 0,
		\end{align*} 
		which contradicts  the fact that $X^{\infty} \cap \mathcal{K}(f) = \{0\}$. 
	\end{proof}

	\begin{lemma}\label{lemma5}
		If $X$ is a convex set and $f$ is a convex function and the following condition holds
		\begin{equation}\label{not-CQ}
			X^{\infty} \cap \mathcal{K}(f) \neq \{0\},
		\end{equation}
		then there exists a sequence $\{u_k\}_{k} \subset \mathbb{R}^n$ with $u_k \rightarrow 0$ as $k \to \infty$  such that $\Sol(u_k) = \emptyset$ for every $k \in\mathbb{N}$.
	\end{lemma}
	
	\begin{proof}
		If \eqref{not-CQ} holds, then there exists a nonzero vector $d \in X^{\infty}$ such that $f^{\infty} (d) \leq 0$. For each $k\in\mathbb{N}$, let $u_k:=\frac{1}{k}d$. Then, $u_k \rightarrow 0$ as $k\to\infty$ and $\langle u_k, d \rangle > 0$ for all $k \in \mathbb{N}$. Hence,  
		$$(f_{u_k})^\infty(d)= f^{\infty} (d) - \langle u_k, d 
		\rangle < 0 \ \  \forall k \in \mathbb{N}.$$
		Fix any $x_0\in X\cap \mathrm{dom}\, f$. It follows from the convexity of $X$ that $x_0+td\in X$ for all $t>0$.  By \cite[Proposition 2.5.2]{AT} and the convexity of $X$ and $f$, one has
		$$(f_{u_k})^\infty(d)=\lim_{t\to\infty}\dfrac{f_{u_k}(x_0+td)-f_{u_k}(x_0)}{t}= \lim_{t\to\infty}\dfrac{f_{u_k}(x_0+td)}{t}.$$
		Hence,
		$$\lim_{t\to\infty}\dfrac{f_{u_k}(x_0+td)}{t}<0.$$
		This implies that $\lim_{t\to\infty} {f_{u_k}(x_0+td)}=-\infty$ for any $k\in\N$. Hence, $\Sol (u_k) = \emptyset$ for all $k\in\mathbb{N}$. The proof is complete.
	\end{proof}	
	
The following theorem presents necessary/sufficient conditions for the 
lower semicontinuity of the  solution  map $\Sol(\cdot)$.
	
\begin{theorem} [The lower semicontinuity of the solution map] \label{theorem12} 
 If the following conditions are satisfied:
 \begin{itemize}
  \item[$(a)$] $\Sol(0)$ is a singleton,
			
  \item[$(b)$] $X^{\infty} \cap \mathcal{K}(f) = \{0\}$, 
  \end{itemize}
 then $\Sol(\cdot)$  is lsc  at $0$. Conversely, if  $\Sol(\cdot)$ is lsc at $0$, then 
 $(a)$ holds true. Moreover, if $X$ and $f$ are additionally assumed to be convex, 
 then $(b)$ is also satisfied.
\end{theorem}
	
\begin{proof}
 Suppose that $(a)$ and $(b)$ hold. By $(a)$, $\Sol(0) = \{\bar x\}$ for some
  $\bar{x} \in X$. Let $U$ be an open neighborhood containing  
 $\bar x$. Then, by (b) and \eqref{equa-6}, there exists $\varepsilon_{1} > 0$ 
 such that 
 $$X^\infty\cap \mathcal{K}(f_u)=\{0\}\ \  \forall u\in \mathbb{B}_{\varepsilon_1}.$$ 
		It follows from Theorem \ref{theo:sta01}(c) that  $\Sol(u) \neq \emptyset$ for every $u$ with $\Vert u \Vert < \varepsilon_1$. Since $\Sol(\cdot) $ is usc at $0$, by Theorem \ref{theo:sta01}(e), there exists $\varepsilon_{2} > 0$ such that $\Sol(u) \subset U$ for every $u$ satisfying $\Vert u \Vert < \varepsilon_{2}$. Thus 
		$\Sol(u) \cap U \neq \emptyset$.
		Therefore, by taking $\varepsilon := \min\{\varepsilon_{1}, \varepsilon_{2}\} > 0$, we obtain that $\Sol(\cdot)$ is lsc at $0$. 	
		
		Conversely, suppose that $\Sol(\cdot)$ is lsc at $0$. We show that the condition $(a)$  is satisfied. Suppose on the contrary that $(a)$ does not hold, i.e., $\Sol(0)$ contains at least two different points $\bar x$ and $ \bar y$.  We choose $\lambda = (\lambda_1, \ldots,\lambda_n) \in \mathbb{R}^n$
		satisfying
		\begin{equation*}
			\Vert \lambda\Vert =1, ~~ - \langle \lambda, \bar x \rangle > - \langle \lambda, \bar y \rangle.
		\end{equation*}
		So, there exists an open set $U$ containing $\bar x$ such that
		\begin{equation}\label{eq:c2}
			- \langle \lambda, x \rangle > - \langle \lambda, \bar y \rangle\ \  \forall  x \in U.
		\end{equation}
		For any $\varepsilon>0$ and $x\in U$,  from  \eqref{eq:c2} it follows that
		$$
		f_{\varepsilon\lambda}(x)= f(x)+  \langle -\lambda\epsilon, x\rangle
		> f(\bar x) +\epsilon \langle -\lambda, \bar y\rangle =f(\bar y) -\epsilon \langle \lambda, \bar y\rangle=	f_{\varepsilon\lambda}(\bar y).$$
		It implies that $x \notin  \Sol(\epsilon\lambda)$. Hence, there exists $\epsilon\lambda\rightarrow 0$ as $\varepsilon\rightarrow0$
		such that $\Sol(\epsilon\lambda)\cap U=\emptyset$, which contradicts that $\Sol(\cdot)$ is lsc at $0$, i.e., condition $(a)$ holds. 
		
Now, suppose in addition that $X$ and $f$ are convex but $(b)$ does not hold. 
Then, by Lemma \ref{lemma5}, there exists a sequence $\{u_k\}_{k} \subset
\mathbb{R}^n$ such that $u_k \rightarrow 0$ as $k \to \infty$ and $\Sol(u_k) 
= \emptyset$ for all $k\in\mathbb{N}$. This  contradicts  that $\Sol(\cdot)$ is 
lsc at $0$.
\end{proof}

\subsection{The continuity of the optimal value function}

	Now, we focus our attention on the continuity of the optimal value function $\mu$.  
	\begin{theorem} \label{theorem22}
		The following assertions hold:
		\begin{enumerate}[$(a)$]
			\item $\mu$ is  usc at $0$.
			\item If $X^{\infty} \cap{K}(f)=\{0\}$,
			then $\mu$ is lsc at $0$ and so it is continuous at  $0$. 
		\end{enumerate}	
	\end{theorem}
	
	\begin{proof}
		$(a)$ Let $\{u_k\}_{k} \subset \mathbb{R}^n$ be a sequence converging to $0$. Since $X\neq\emptyset$, we have $\mu (0)< +\infty$. Then, there is a sequence $\{x_l\}_{l}$ in $\mathbb{R}^n$ such that $x_l\in X$ and $f(x_l) \to \mu (0)$ as $l \to \infty$. For each $l \in \mathbb{N}$, we have $\mu(u_k) \leq f_{u_k}(x_l)$ for all   $k \in \mathbb{N}$. This implies that $$\limsup\limits_{k\to \infty}\mu(u_k) \leq f(x_{l}).$$ 	
		Taking $l \to \infty$  we get	
		$$\limsup\limits_{k\to \infty}\mu(u_k) \leq \mu(0).$$ 
		This means that $\mu$ is usc at $0$. 
		
		$(b)$ Let $\{u_k\}_{k} \subset \mathbb{R}^n$ be an arbitrary sequence converging to $0$. We show that 
		$$\liminf\limits_{k\to \infty}\mu(u_k) \geq \mu(0),$$
		i.e., $\mu$ is lsc at $0$. Suppose on the contrary that
		$\liminf\limits_{k\to \infty}\mu(u_k) < \mu(0)$.
		By taking a subsequence if
		necessary, we can assume  that
		$$\liminf\limits_{k\to \infty}\mu(u_k) =\lim\limits_{k\to \infty}\mu(u_k).$$
		Then there exist $k_0 \in \mathbb{N}$ and $\beta \in \mathbb{R}$ such that $\beta < \mu(0)$ and $\mu(u_k)\leq \beta$ for all $k \geq k_0$.
		Since $u_k\to 0$ as $k\to \infty$ and Theorem \ref{theo:sta01}(c), there exists an integer $k_1 \geq k_0$ such that  $\Sol(u_k)\neq \emptyset$ for all $k\geq k_1$. For  each $k\geq k_1$, take any $x_k\in \Sol(u_k)$. Then we have $x_k\in X$ and
		\begin{equation}\label{42}
			f_{u_k}(x_k)=\mu(u_k)\leq \beta.
		\end{equation}
		We show that the sequence $\{x_k\}_{k}$ is bounded. Indeed, if otherwise, then, without loss of generality, we can assume that $\Vert x_k\Vert \neq 0$ for all $k\geq k_1$ and $\Vert x_k \Vert \to +\infty$ as $k\to \infty$. Then, the sequence 
		$\{\|x_k\|^{-1}x_k\}_{k}$ is bounded, and hence it has a convergent subsequence. Without loss of generality, we can assume that this sequence itself converges to some $d\in \mathbb{R}^n$ with $\Vert d\Vert =1$ and $d\in X^{\infty}$. Let $t_k:=\|x_k\|$ and 
		$d_k:=t_k^{-1}x_k.$ By \eqref{42}, we obtain
$$\frac{f_{u_k}(t_kd_k)}{ t_k}\leq \frac{\beta}{t_k}\ \  \forall  k \geq k_1.$$
Letting $k\to \infty$ in the last inequality we get
		\begin{align*}
			f^\infty(d) & \leq \liminf_{k \rightarrow \infty} \dfrac{f(t_k d_k)}{t_k} = \liminf_{k \rightarrow\infty}\dfrac{f(t_kd_k)-\langle u_k, d_k\rangle}{t_k} = \liminf_{k \rightarrow \infty} \dfrac{f_{u_k}(x_k)}{t_k} \leq 0.
		\end{align*} 
		This with $d\in X^{\infty}$ contradict  the assumption that $X^{\infty} \cap{K} (f)=\{0\}$.  Therefore, $\{x_k\}_{k}$ is bounded. 
		
		Now, passing a subsequence if needed, $\{x_k\}_{k}$ converges to some $\hat{x} \in X$.
		By \eqref{42} and the lower semicontinuity of $f$, we obtain 
		$$f(\hat{x})\leq  \lim\limits_{k\to \infty} f_{u_k} (x_k) \leq \beta.$$
		Combining this with $\beta < \mu(0)$, we have $f(\hat{x})  < \mu(0)$, which contradicts that $\hat{x}\in X$. Therefore, $\mu$ is lsc at $0$ and the proof is complete.
	\end{proof}

\section{The Quasiconvex Case}\label{section5}
	
In this section, we apply our previous results to the particular case when the 
objective function in problem \eqref{problem} is quasiconvex or $\alpha$-robustly
quasiconvex (see Definition \ref{alpha:robust}).

As mentioned in the introduction, when the function $f$ is nonconvex, then the 
usual asymptotic function $f^{\infty}$ does not provide adequate information
on the behavior of $f$ at infinity. For instance, and in relation to Theorem 
\ref{weak-sharp-Thrm}, we mention that when the function $f$ is quasiconvex, 
the assumption 
$$X^{\infty} \cap \mathcal{K}(f) =  \{0\},$$
is too restrictive. Indeed, let us consider the one-dimensional real-valued 
function $f: \mathbb{R} \rightarrow \mathbb{R}$ given by $f(x) = \sqrt{\lvert 
x \rvert}$ and $X=\mathbb{R}$. Here $f$ is coercive and ${\rm argmin}_{X}\,f 
= \{0\}$ is a singleton, but $X^{\infty} \cap \mathcal{K}(f) = \mathbb{R}$, 
and Theorem \ref{weak-sharp-Thrm} cannot be applied even in this basic situation.
	
On the other hand, if we use any of the generalized asymptotic functions
$f^{\infty}_{q}$ and $f^{\infty}_{\lambda}$, we obtain that
$$f^{\infty}_{q} (u) = + \infty\ \  \forall  u \neq 0, ~~ {\rm and}~~ 0 < f^{\infty}_{\lambda} (u) \leq 2\ \  \forall   u \neq 0.$$
Therefore, 
$$X^{\infty} \cap K_{q} (f) =  \{0\}, ~~ {\rm and}~~ X^{\infty} \cap 
K_{\lambda} (f) =  \{0\},$$
where $K_{q}(f) := \{d \in \mathbb{R}^{n}: \, f^{\infty}_{q} (d) \leq 0\}$ and
$K_{\lambda}(f) := \{d \in \mathbb{R}^{n}: \, f^{\infty}_{\lambda} (d) \leq 0\}$,
respectively. 
	
Before stating main results of this section, we give a result on the nonemptiness 
and compactness of the solution set to constrained optimization problems by 
using the $q$-asymptotic function.

\begin{lemma}\label{lemma-51} 
 Assume that $X \subset \mathrm{dom}\, f$ is a convex set and $f$  is a 
 quasiconvex function. Then, the  fo\-llo\-wing assertions are equivalent:
 \begin{enumerate}[$(a)$]
  \item $\Sol\eqref{problem}$ is nonempty and compact.
  
  \item $X^\infty\cap \mathcal{K}_q(f)=\{0\}$.
 \end{enumerate}
\end{lemma} 

\begin{proof} 
 By \cite[Theorem 4.7]{FFB-Vera}, it suffices to   show that $(b)$ is equivalent to 
 the following condition
		\begin{equation*} 
			\mathcal{K}_q(f+\delta_X)=\{0\},
		\end{equation*} 
		where $\delta_X$ is the indicator function of $X$ and defined by
		\begin{equation*}
			\delta_X=
			\begin{cases}
				0, &\text{if}\ \ x\in X,
				\\
				+\infty, &\text{otherwise}.
			\end{cases}
		\end{equation*}
Indeed, by definition of the $q$-asymptotic function and the convexity of $X$, we have
		\begin{align*}
			(f+\delta_X)_q^\infty(d)&=\sup_{x \in {\rm dom}\,(f+\delta_X)} \sup_{t>0} \frac{(f+\delta_X)(x+td) - (f+\delta_X)(x)}{t}
			\\
			&=\sup_{x \in X} \sup_{t>0} \frac{f(x+td)+\delta_X(x+td) - f(x)}{t}
			\\
			&=
			\begin{cases}
				f_q^\infty(d), &\text{if}\ \ d\notin X^\infty,
				\\
				+\infty, &\text{otherwise},
			\end{cases}
			\\
			&=f_q^\infty(d)+\delta_{X^\infty}(d).
		\end{align*}
Hence, $X^\infty\cap \mathcal{K}_q(f)=\{0\}$ if and only if $\mathcal{K}_q
(f+\delta_X)=\{0\}$, which completes the proof. 
\end{proof}
 
\begin{remark}\rm 
 The equivalent conditions $(a)$ and $(b)$ in Lemma \ref{lemma-51} do not imply 
 the coercivity of $f$ on $X$. For example, let $f(x)=\tfrac{x^2}{1+x^2}$ and
 $X=\mathbb{R}$. Clearly, $f$ is quasiconvex on $\mathbb{R}$ and
 $\mathcal{K}_q(f)=\{0\}$. Hence, condition $(b)$ in Lemma \ref{lemma-51} 
 is satisfied, but the function $f$ is not coercive on $\mathbb{R}$. 
\end{remark}
 
In the same spirit than Proposition \ref{prop:comparison2}, the following 
corollary shows that condition \eqref{equa-CQq} is weaker than \eqref{Son-CQ} 
for quasiconvex functions, too.

\begin{corollary} 
Assume that $X\subset \mathrm{dom}\,f$ and it is   convex. If the function $f$ is  quasiconvex      and bounded from below on $X$ and   \eqref{Son-CQ} holds, then so is \eqref{equa-CQq}. 
\end{corollary}

\begin{proof} 
 By \eqref{Son-CQ} and
 \cite[Theorem 6.4]{Kim-Tung-Son-23}, $\Sol\eqref{problem}$ is nonempty and
 compact. Thus, the desired conclusion follows directly from Lemma \ref{lemma-51}. \end{proof}
	
In the next proposition we improve Theorem \ref{weak-sharp-Thrm} for proper, 
lsc, $\alpha$-robustly quasiconvex functions. 
	
\begin{proposition}\label{sta:robust}
 Let $X$ be a convex set and $f$ be an $\alpha$-robustly quasiconvex function
 $(\alpha>0)$ with $X \subset  {\rm dom}\,f$. If 
 \begin{equation}\label{equa-CQq}
  X^{\infty} \cap K_{q} (f) = \{0\},
 \end{equation}
 then there exists $\varepsilon > 0$ such that for all $u \in \mathbb{B}_{\varepsilon}$,
 the following statements hold:
 \begin{enumerate} [$(a)$]
  \item $f_u$ is bounded from below on $X$.
			
  \item ${\rm \Sol}(u)$ is nonempty and compact.
			
  \item $ \limsup_{u \rightarrow 0} {\rm \Sol} (u) \subset {\rm \Sol} (0)$.
  
  \item $\Sol(\cdot)$ is usc at $0$.
 \end{enumerate}
\end{proposition}
	
	\begin{proof}
		We first show that there exists $\varepsilon > 0$ such that  the following condition holds  
		\begin{equation}\label{q:assump}
			X^{\infty} \cap K_{q} (f_{u}) = \{0\} \ \ \forall u \in 
			\mathbb{B}_{\varepsilon}.
		\end{equation}
		Indeed, suppose on the contrary that for every $k \in \mathbb{N}$, there 
		exists $u_{k} \in \mathbb{B}_{\frac{1}{k}}$ such that $X^{\infty} \cap 
		K_{q} (f_{u_{k}}) \neq \{0\}$, i.e., there exists $d_{k} \in X^{\infty} 
		\backslash \{0\}$ such that 
		$$(f_{u_{k}})^{\infty}_{q} (d_{k}) \leq 0.$$
		Since $X^{\infty}$ is a closed cone, by passing a subsequence if necessary, we 
		may assume that $\left\{h_{k} := \frac{d_{k}}{\lVert d_{k} \rVert}\right\}_{k}
		\subset X^{\infty}$ converges to $h \in X^{\infty}$ with $\lVert h \rVert = 1$.
		
		For every $k \in \mathbb{N}$, since $f^{\infty}_{q}$ is positively homogeneous of degree one and $(f_{u_{k}})^{\infty}_{q} (h_{k}) \leq 0$, and by using relation \eqref{equa-5q},
		\begin{align*}
			 (f_{u_{k}})^{\infty}_{q} (h_{k}) 	= f^{\infty}_{q} (h_{k}) - \langle u_{k},  h_{k} \rangle\leq 0. 
		\end{align*}
		Since $f$ is lsc, $f^{\infty}_{q}$ is lsc too by \cite[p. 118]{FFB-Vera}, thus
		$$f^{\infty}_{q} (h) \leq \liminf_{k \rightarrow + \infty} f^{\infty}_{q} (h_{k}) \leq 
		\liminf_{k \rightarrow + \infty} \, \langle u_{k},  h_{k} \rangle = 0.$$
		Hence, $h \in X^{\infty} \cap K_{q} (f)$ with $h \neq 0$, a contradiction. 
		Therefore, relation \eqref{q:assump} holds.
		
		Since $X$ is convex and $f$ is $\alpha$-robustly quasiconvex, $f_{u}(x) = f(x) - \langle u, x \rangle$ is quasiconvex for all $u\in\mathbb{B}_\alpha$. Let $\varepsilon$ be satisfied \eqref{q:assump} and $\varepsilon<\alpha$. Hence, $(a)$ and $(b)$ follow from relation
		\eqref{q:assump} and Lemma \ref{lemma-51}. 
		
		The proof of part $(c)$ is quiet similar to the proof of Theorem \ref{theo:sta01}(d), so omitted.
		
		$(d)$: Suppose on the contrary that  $\Sol(\cdot)$ is not usc at $0$. Analysis similar to that in the proof of Theorem \ref{theo:sta01}(e) shows that  there exist   sequences $u_k$ and $x_k$ such that  $u_k\rightarrow 0$,  $x_k \in \Sol(u_k)$ for all $k \in \mathbb{N}$, and the sequence $\{x_k\}$ is unbounded. Without loss of generality, we can assume that $\|x_k\| \to + \infty$ and $d_k := \frac{x_k}{\|x_k\|} \to d \in X^{\infty}$  with $\Vert d \Vert = 1$. 
		Since $u_{k} \to 0$, there exists $k_{1} \in \mathbb{N}$ such that 
		$u_{k} \in \mathbb{B}_{\alpha}$ for all $k \geq k_{1}$. Furthermore, since $\|x_k\| \to + \infty$  for every $t>0$, there exists $k_{2} \in \mathbb{N}$ such that $0 < \frac{t}{\|x_k\|} < 1$ for all $k \geq k_{2}$.
		
		Take any $y \in X$, then $f_{u_k} (x_k) \leq f_{u_k}(y)$ for all $k \in \mathbb{N}$. Since $X$ is convex and $f$ is $\alpha$-robustly quasiconvex, we obtain for every $k \geq k_{0} := \max\{k_{1}, k_{2}\}$ that
		\begin{align*}
			f\left( \left(1 - \frac{t}{\|x_k\|}\right) y + \frac{t}{\|x_k\|} x_{k} \right) - \left\langle u_{k}, \left(1 - \frac{t}{\|x_k\|}\right) y + \frac{t}{\|x_k\|} x_{k}  \right\rangle
			& \leq \max\{f_{u_{k}} (x_{k}),  f_{u_{k}} (y)\} \\
			& = f (y) - \langle u_{k}, y \rangle.
		\end{align*}
		Since $f$ is lsc, we have
		$$f(y + td) - 0 \leq f(y) - 0\ \  \forall ~ t > 0\ \  \forall  y \in X.$$
		This implies 
		$f^{\infty}_{q} (d) \leq 0.$
		Therefore, $d \in X^{\infty} \cap 
		K_{q} (f)$, a contradiction.
	\end{proof}

Furthermore, we also adapt the results for solution stability for the quasiconvex 
case below.
	
\begin{proposition} [The lower semicontinuity of the solution map] \label{prop12q} 
 Let $X$ be a convex set and $f$ be an $\alpha$-robustly quasiconvex ($\alpha 
 > 0$) with $X \subset {\rm dom}\,f$. Then $\Sol(\cdot)$ is lsc at $0$ if and 
 only if the following conditions hold:
 \begin{itemize}
  \item[$(a)$] $\Sol(0)$ is a singleton;
			
  \item[$(b)$] $X^{\infty} \cap K_{q}(f) = \{0\}$. 
 \end{itemize}
\end{proposition}
 
\begin{proof}
 The proof of the part ``only if'' is similar to the proof of first one of Theorem 
 \ref{theorem12}, so omitted.
		
 We now assume that $\Sol(\cdot)$ is lsc at $0$. Then, by Theorem 
 \ref{theorem12}, (a) holds true. Assume that  $\Sol(0)=\{\bar x\}$. If (b) 
 does not hold, then there exists $d\in X^\infty\setminus\{0\}$ such that
 $f^\infty_q(d)\leq 0$. Clearly, $\bar x+td\in X$ for all $t>0$. Let $u_k: = 
 \frac{1}{k}d$. Then $u_k\to 0$ as $k\to\infty$ and 
		$$(f_{u_k})^\infty_q(d)=f^\infty_q(d)-\langle u_k, d\rangle=f^\infty_q(d)-\frac{1}{k}\|d\|^2<0.$$
		By definition, we have
		\begin{equation*}
			\sup_{t>0}\frac{f_{u_k}(\bar x+td)-f_{u_k}(\bar x)}{t}<0.   
		\end{equation*}
		This implies that $f_{u_k}(\bar x+td)\to-\infty$ as $t\to\infty$ and $\Sol(u_k)=\emptyset$. Hence, $\Sol(\cdot)$ is not lsc at $0$, a contradiction.
	\end{proof}	
	As a consequence, we have the following corollary.
	
\begin{corollary} [The continuity of the solution mapping] \label{Stability3q}  
 Let $X$ be a convex set and $f$ be an $\alpha$-robustly quasiconvex ($\alpha 
 > 0$) with $X \subset {\rm dom}\,f$. Then $\Sol(\cdot)$  is continuous at $0$ 
 if and only if   conditions $(a)$ and $(b)$ in Proposition \ref{prop12q} hold.	
\end{corollary}
	
Finally, we ensure the continuity of the value function $\mu$ at $0$ when $f$ 
is proper, lsc and $\alpha$-robustly quasiconvex $(\alpha > 0)$.

\begin{proposition}[The continuity of the optimal value function] \label{theorem22q}
 Let $X$ be a convex set and $f$ be an $\alpha$-robustly quasiconvex ($\alpha 
 > 0$) with $X \subset {\rm dom}\,f$. If $X^{\infty} \cap{K}_{q} (f)=\{0\}$, 
 then $\mu$ is  continuous at $0$.
\end{proposition}
	
	\begin{proof}
		We repeat the proof of Theorem \ref{theorem22} until relation \eqref{42}. Hence, we need to show
		that the sequence $\{x_k\}_{k \geq k_1}$ is bounded, where $x_k\in\Sol(u_k)$ for all $k\geq k_1$ with $u_k$ being an arbitrary sequence converging to $0$. Indeed, if $\{x_k\}_{k\geq k_1}$  is unbounded, then, without loss of generality, we assume that $\Vert x_k \Vert \to +\infty$. Hence, 
		$\{\frac{x_k}{\lVert x_{k} \rVert}\}_{k}$ is bounded. Passing a subsequence if needed, $\frac{x_k}{\lVert x_{k} \rVert} \rightarrow d \in X^{\infty}$  with $\Vert d \Vert =1$. 
		
		Since $u_{k} \to 0$, there exists $k_{2} \in \mathbb{N}$ such that 
		$u_{k} \in \mathbb{B}_{\alpha}$ for all $k \geq k_{2}$. Furthermore, since $\|x_k\| \to + \infty$, for every $t>0$, there exists $k_{3} \in \mathbb{N}$ such that $0 < \frac{t}{\|x_k\|} < 1$ for all $k \geq k_{3}$.
		Take any $y \in X$. Since $x_{k} \in \Sol(u_{k})$, $f_{u_k} (x_k) \leq f_{u_k}(y)$ for all $k \in \mathbb{N}$. Since $X$ is convex and $f$ is $\alpha$-robustly quasiconvex, we have $f^{\infty}_{q} (d) \leq 0$ by the same analysis than Proposition \ref{sta:robust}. Hence, $d \in X^{\infty} \cap 
		K_{q} (f)$, a contradiction. 
		
		Therefore, $\{x_k\}_{k}$ is bounded and the rest of the proof follows as in Theorem \ref{theorem22}. We note here that the upper semicontinuity of $\mu$ at $0$ does not depend on the assumptions on the function $f$ and $X$.
	\end{proof}

\section{Conclusions}

We contributed to the discussion on the analysis of optimization problems at 
infinity by studying sufficient conditions for weak sharp minima and solution 
stability in the general nonconvex case and also in the quasiconvex case. 

By using the usual tools from (generalized) asymptotic analysis, we proved 
that this classical approach is finer than the one from the variational analysis
at infinity, specially for weak sharp minima since the sufficient conditions based 
on asymptotic tools are weaker than the ones based on variational analysis 
at infinity for the whole class of proper, lsc and convex functions (and also for 
the quasiconvex ones).

In this sense, we strongly believe that the valuable efforts for obtaining 
useful information for optimization problems from the infinity could be 
improved, and this improvements should be based on adding a ``direction 
term'' in   Definitions \ref{def31} and \ref{def41}, that is, by defining the
{\it normal cone of the set $X$ at the infinity in the direction $v$} and its 
respectively {\it limiting subdifferential of $f$ at the infinity in the direction 
$v$}. This will be a matter of a subsequent work.

We hope that our results could provide new lights for further developments 
on asymptotic analysis and variational analysis at infinity.
\bigskip

\noindent \textbf{Acknowledgements} A part of this paper was completed when the  authors were visiting the 
 Vietnam Institute for Advanced Study in Mathematics (VIASM) in 2024. The 
 authors would like to thank VIASM for their financial support and hospitality. 
 \medskip
 
 \noindent \textbf{Funding} The first author was supported by  ANID--Chile under project Fondecyt  Regular 1241040. The second and the third authors were supported by  Vietnam National Foundation for Science and
 Technology Development (NAFOSTED) under grant number 101.01-2023.23. 
 \medskip
  
 \noindent \textbf{Declarations} 
 \\
 \noindent{\bf Data availability} The authors confirm that the data supporting the findings of this study are available within the article.
 \\
 {\bf Conflict of interest} The authors declare that they have no conflict of interest.

	\bibliographystyle{amsplain}
	%\bibliography{D:/Submission/BibPureMath1,D:/Submission/BibAppMath1}

\end{document}